\documentclass[letterpaper, 10 pt, conference]{ieeeconf} 
\IEEEoverridecommandlockouts
\usepackage[utf8]{inputenc}
\usepackage{CJKutf8}
\usepackage{ascmac}
\usepackage{url}
\usepackage{amsmath}
\usepackage{graphicx}
\usepackage{amsfonts}
\usepackage{bm}
\newtheorem{definition}{Definition}
\newtheorem{theorem}{Theorem}
\newtheorem{lem}{Lemma}
\newtheorem{prop}{Proposition}
\newtheorem{cor}{Corollary}
\newtheorem{prob}{Problem}

\usepackage{color}

\title{\LARGE \bf
A Controllability Gramain Shaping with LMI Constraints \\under Bures--Wasserstein Distance
}

\ifdefined\iflatexml
\author{Koju Nishimotot%
\thanks{\* This work has been submitted to the IEEE for possible publication. Copyright may be transferred without notice, after which this version may no longer be accessible.}%
\thanks{
Koju Nishimoto is with the Institute of Technology, Shimizu Corporation, Tokyo 135-0044, Japan {\tt\small koju.nishimoto@shimz.co.jp}}%
,
Yuki Onishi%
\thanks{
Yuki Onishi is with the National Institute of Informatics, Chiba 277-0882, Japan {\tt\small onishi@nii.ac.jp}}%
,
Riku Funada%
\thanks{
Riku Funada is with the Graduate School of Informatics, Kyoto University, Kyoto 606-8501, Japan {\tt\small funada@i.kyoto-u.ac.jp}}%
,
Mitsuji Sampei%
\thanks{
Mitsuji Sampei is with the Polytechnic University of Japan, Tokyo 187-0035, Japan {\tt\small mitsuji@sampei.jp}}
}
\else
\author{Koju Nishimoto, Yuki Onishi, Riku Funada and Mitsuji Sampei% <-this % stops a space
\thanks{* This work has been submitted to the IEEE for possible publication. Copyright may be transferred without notice, after which this version may no longer be accessible.}%
\thanks{Koju Nishimoto is with the Institute of Technology, Shimizu Corporation, Tokyo 135-0044, Japan {\tt\small koju.nishimoto@shimz.co.jp}}%
\thanks{Yuki Onishi is with the National Institute of Informatics, Chiba 277-0882, Japan {\tt\small onishi@nii.ac.jp}}%
\thanks{Riku Funada is with the Graduate School of Informatics, Kyoto University, Kyoto 606-8501, Japan {\tt\small funada@i.kyoto-u.ac.jp}}
\thanks{Mitsuji Sampei is with the Polytechnic University of Japan, Tokyo 187-0035, Japan {\tt\small mitsuji@sampei.jp}}
}
\fi

\begin{document}

\maketitle
\thispagestyle{empty}
\pagestyle{empty}

\begin{abstract}
This paper proposes a controller design method for shaping the controllability Gramian into a desired form to design the effect from exogenous inputs to the system state.
Using the Bures--Wasserstein distance, we formulate the shaping problem as the minimization of the distance between the system Gramian and a desired Gramian, and the objective function is shown to be strictly convex on the set of symmetric positive definite matrices. 
In addition, by deriving a semidefinite programming formulation via a linear matrix inequality (LMI), computational efficiency is improved and additional LMI constraints can be incorporated.
When the exogenous input is modeled as Gaussian white noise, the proposed framework is closely related to $H_2$ control, which can be interpreted as a special case of optimal transport.
Numerical examples demonstrate anisotropic controllability design for a guidance robot and verify the ability to impose additional directional constraints through LMIs.
The numerical examples also confirm that the proposed method approaches $H_2$ control as the desired Gramian tends to zero.
\end{abstract}

\section{Introduction}

Designing systems that are easy to control is important for smooth operation.
However, the ``ease of control'' depends on the objective and application, and is not easy to evaluate in a general framework.
Since many control problems have been formulated mathematically in control theory, defining ``ease of control'' within this framework is expected to provide a general measure for analysis and design.
In classical control, input-output properties and disturbance responses are evaluated through the frequency responses of transfer and sensitivity functions \cite{Feedback_Control_of_Dynamic_Systems}.
For nonlinear systems, application-dependent measures are also used, such as manipulability ellipsoids in robotics \cite{Modern_Robotics} and design indices for drones considering fault tolerance and hovering performance \cite{Mochida2022}.

In systems driven by exogenous inputs, it is important to design not only stability and tracking performance but also the response to such inputs.
In particular, it is important to specify the directions in the state space along which the system is easy or difficult to move.
This is particularly important in human-interactive systems, where operability and stability must be balanced.
Impedance control \cite{Hogan1984} and admittance control \cite{Landi2017} shape the relation between external forces and motion, but mainly prescribe input-output relations or local responses, rather than directional reachability over the entire state space.

The controllability Gramian \cite{Kalman1963,Sato2024,Lygeros2016,Amita2019} is suitable for this purpose because it characterizes reachability in terms of input energy.
By shaping the controllability Gramian, one can quantitatively design the apparent controllability from exogenous inputs and quantify both ease and difficulty of motion in each state direction.
This idea is expected to be applicable to systems such as guidance robots for visually impaired users \cite{Kayukawa2022,Takagi2025} and manipulators with direct teaching \cite{Bascetta2013,Matteo2016}.
The authors have previously proposed Gramian shaping for such systems from the user's perspective \cite{Nishimoto2024}.
In the Gramian shaping, a desired Gramian is specified, and the controllability Gramian from the exogenous input is shaped by an internal feedback input so as to approach the desired one.
Through the anisotropy of the Gramian, this approach enables quantitative design of ease and difficulty of control in each state direction.

To make such reachability design viable, the distance between the system Gramian and the desired Gramian should yield a tractable optimization problem to admit a clear interpretation.
In the previous Gramian shaping method \cite{Nishimoto2024}, this discrepancy was measured by the affine-invariant Riemannian (AIR) distance \cite{Positive_Definite_Matrices}.
However, AIR-based Gramian shaping leads to a nonconvex optimization problem, so global optimality is not guaranteed and efficient algorithms are not readily available.
As a result, robustness and computational efficiency are difficult to ensure.
Moreover, the AIR distance has no clear physical interpretation, making it difficult to relate Gramian shaping to existing control frameworks.

To overcome these limitations, this paper proposes a new Gramian shaping method based on the Bures--Wasserstein distance (BW distance).
The BW distance is defined on symmetric positive-semidefinite matrices \cite{Oostrum2022} and coincides with the $2$-Wasserstein distance between Gaussian distributions with the same mean \cite{BHATIA2019165}.
Thanks to this property, it has been used in graph generation \cite{Keyue2026}, statistical theory based on BW barycenters \cite{Haasler2024,Adam2024,Leonardo2024}, and geodesic analysis \cite{Thanwerdas2023}.

The main contributions of this paper are summarized as follows.
\begin{enumerate}
\item We introduce the BW distance into controllability Gramian shaping from exogenous inputs to make the optimization problem convex on the set of symmetric positive definite matrices.
Thus, BW-based Gramian shaping can be formulated as a convex optimization problem.
\item Using the linear matrix inequality (LMI) representation of the BW distance, we derive a semidefinite programming (SDP) formulation of Gramian shaping, which allows various additional LMI constraints.
\item We clarify the relation between the proposed method and $H_2$ control, and show that $H_2$ control can be interpreted as a special case of BW-based Gramian shaping.
\end{enumerate}

This paper is organized as follows.
Section~II reviews the controllability Gramian, introduces the Gramian for exogenous inputs, and summarizes the relation between realizable  Gramians and feedback gains.
Section~III defines the proposed method, namely, BW-based Gramian shaping and shows that it can be formulated as a convex problem.
Section~IV presents an SDP-based solution using an LMI representation of the BW distance.
Section~V discusses the relation between the proposed method and $H_2$ control, showing that the proposed framework provides a link between classical optimal control and optimal transport.
Section~VI verifies the effectiveness of the proposed method through numerical examples.

Notation: Vectors are denoted by lowercase bold letters such as $\bm{a}$, and matrices by uppercase bold letters such as $\bm{A}$.
$\bm{I}_n$ and $\bm{O}_n$ denote the $n\times n$ identity matrix and zero matrix, respectively.
$\bm{O}_{n\times m}$ denotes the $n\times m$ zero matrix, and $\bm{0}_n$ denotes the $n$-dimensional zero vector.
$\bm{A}^\top$ and $\bm{A}^H$ denote the transpose and Hermitian transpose of $\bm{A}$, respectively.
$\bm{A}^\dagger$ denotes the Moore--Penrose pseudoinverse of $\bm A$.
$\mathbb{R}^{n}$ and $\mathbb{R}^{n\times m}$ denote the set of $n$-dimensional real vectors and the set of $n\times m$ real matrices, respectively.
$\mathbb{R}_+$ denotes the set of nonnegative real numbers.
$\mathbb{S}_{+}^n$ denotes the set of $n\times n$ symmetric positive-semidefinite matrices, and $\mathbb{S}_{++}^n$ denotes the set of $n\times n$ symmetric positive definite matrices.
$\mathcal{O}(n)$ and $\mathcal{SO}(n)$ denote the orthogonal group and the special orthogonal group in dimension $n$, respectively.
A diagonal matrix with diagonal entries $a_1,a_2,\ldots,a_n$ is denoted by $\mathrm{diag}\left(a_1,a_2,\ldots,a_n\right)$.
A block diagonal matrix formed by the matrices $\bm{A}_1,\bm{A}_2,\ldots,\bm{A}_n$ is denoted by $\mathrm{blkdiag}\left(\bm{A}_1,\bm{A}_2,\ldots,\bm{A}_n\right)$.
For $\bm{A}\in\mathbb{S}_{++}^n$, there exists a spectral decomposition $\bm{A}=\bm{U}\bm{\Lambda}\bm{U}^\top$, where $\bm{\Lambda}=\mathrm{diag}(\lambda_1,\lambda_2,\ldots,\lambda_n)$, $\lambda_i> 0$ for all $i\in\{1,2,\ldots,n\}$, and $\bm{U}\in\mathcal{O}(n)$.
For a real number $\alpha$, the matrix power of a symmetric positive definite matrix is defined by $\bm{A}^\alpha=\bm{U}\bm{\Lambda}^\alpha\bm{U}^\top$, where $\bm{\Lambda}^\alpha=\mathrm{diag}(\lambda_1^\alpha,\lambda_2^\alpha,\ldots,\lambda_n^\alpha)$.

\section{Preliminaries}
In this section, we first review the standard controllability Gramian for stable systems.
We then introduce the controllability Gramian considered in this paper for systems driven by exogenous inputs.

Consider the following linear system with state $\bm{x}\in\mathbb{R}^n$ and input $\bm{u}\in\mathbb{R}^m$:
\begin{align}
  \dot{\bm{x}}&=\bm{A}\bm{x}+\bm{B}\bm{u},\label{eq:linear_system}
\end{align}
where $\bm{A}\in\mathbb{R}^{n\times n}$ and $\bm{B}\in\mathbb{R}^{n\times m}$.

\subsection{Definition of the controllability Gramian}  
The controllability Gramian has been used as a measure describing the influence of inputs on the system state and as a tool for analyzing properties of systems.
For linear systems, the standard controllability Gramian is defined as follows.
\begin{definition}[Ch.~6\cite{Linear_System_Theory_and_Design}]
The controllability Gramian of system \eqref{eq:linear_system} is defined, when $\bm A$ is Hurwitz stable, by
\begin{align}
  \bm{P}_0=&\int_{-\infty}^{0}~e^{-\bm{A}\tau}\bm{B}\bm{B}^\top e^{-\bm{A}^\top\tau}~d\tau\nonumber\\
     =&\int_{0}^{\infty}~e^{\bm{A}\tau}\bm{B}\bm{B}^\top e^{\bm{A}^\top \tau}~d\tau.\label{eq:infinite_controllability_Gramian}
\end{align}
\end{definition}

\begin{cor}[Th.~6.1\cite{Linear_System_Theory_and_Design}]
    If the system is $(\bm A,\bm B)$-controllable and $\bm A$ is Hurwitz stable, then the controllability Gramian is the unique symmetric positive definite solution of the following Lyapunov equation:
\begin{align}
  \bm{O}_n=\bm{A}\bm{P}_0+\bm{P}_0\bm{A}^\top +\bm{B}\bm{B}^\top .
\end{align}
\end{cor}

If the system is unstable, that is, if $\bm A$ has an eigenvalue with a nonnegative real part, the controllability Gramian diverges and is therefore not defined in general.

The controllability Gramian $\bm{P}_0$ corresponds to the minimum input energy: the minimum energy required to reach a state $\bm{x}_0$ from the origin is given by $\bm{x}_0^\top\bm{P}_0^{-1}\bm{x}_0$.
Therefore, the set $\{\bm{x}\mid\bm{x}^\top\bm{P}_0^{-1}\bm{x}\leq1\}$ represents the reachable region within unit input energy, and the anisotropy of $\bm{P}_0$ characterizes the relative controllability degree in each state direction.

\subsection{Controllability Gramian from exogenous inputs}
In this subsection, we define the controllability Gramian for systems subject to exogenous inputs.
Here, $\bm{u}$ is regarded as an internal control input, and we introduce an exogenous input $\bm{v}\in\mathbb{R}^r$ so that the system is described by
\begin{align}\label{eq:outer_input_sys}
    \dot{\bm{x}}&=\bm{A}\bm{x}+\bm{B}\bm{u}+\bm{D}\bm{v},
\end{align}
where $\bm{D}\in\mathbb{R}^{n\times r}$. 
Both $(\bm{A},\bm{B})$ and $(\bm{A},\bm{D})$ are assumed controllable.
The input $\bm{v}$ is exogenous.
Hence, the energy sources of $\bm u$ and $\bm v$ are different, and we primarily consider systems controlled externally by a user such as a human.
To stabilize this system, we apply the feedback control law $\bm{u}=\bm{K}\bm{x}$, which yields
\begin{align}\label{eq:stab_syst}
  \dot{\bm{x}}=(\bm{A}+\bm{B}\bm{K})\bm{x}+\bm{D}\bm{v}.
\end{align}

Since the system is stabilized, the controllability Gramian from the new input $\bm v$ to the state $\bm{x}$ in \eqref{eq:stab_syst} can be redefined as follows.
\begin{definition}[Sec.~2~\cite{Nishimoto2024}]\label{def:unsta_Gram}
    The controllability Gramian of system \eqref{eq:stab_syst} is defined, as a function of the gain $\bm{K}$, by
    \begin{align} \label{eq:unsta_Gram}
      \bm{P}=\int_{0}^{\infty}~e^{(\bm{A}+\bm{B}\bm{K})\tau}{\bm{D}\bm{D}^\top} e^{(\bm{A}+\bm{B}\bm{K})^\top \tau}~d\tau.
    \end{align}
\end{definition}
\begin{cor}
From the relationship between the controllability Gramian and the Lyapunov equation, the following equation holds:
\begin{align}
  (\bm{A}+\bm{B}\bm{K})\bm{P}+\bm{P}(\bm{A} + \bm{B}\bm{K})^\top + {\bm{D}\bm{D}^\top}=\bm{O}_n.\label{eq:extended_lyapnov_equation}
\end{align}
\end{cor}
By redefining the controllability Gramian in the form \eqref{eq:unsta_Gram}, the controllability Gramian of the system now depends on the feedback gain $\bm K$.
Hence, controllability can be modified by adjusting $\bm K$.
The realizable controllability Gramians are restricted by the system structure, namely $\bm{A}$, $\bm{B}$, and $\bm{D}$.
This relation is characterized by the Lyapunov equation \eqref{eq:extended_lyapnov_equation}.
Although \eqref{eq:extended_lyapnov_equation} determines $\bm{P}$ and $\bm{K}$ simultaneously, it can be decomposed into conditions on $\bm{P}$ and formulas for $\bm{K}$ separately by Lemma~\ref{lem:gramian_theorem} and Lemma~\ref{lem:gain_derive_for_gramian} in Appendix.
Lemma~\ref{lem:gramian_theorem} gives an equality condition characterizing the realizable Gramians of system \eqref{eq:outer_input_sys}.
Lemma~\ref{lem:gain_derive_for_gramian} then provides a feedback gain that achieves the realizable Gramian.
Thus, Gramian shaping can be addressed in two separate steps: Gramian design and gain derivation.

The objective of this paper is to use the controllability Gramian in controller design so as to realize desired controllability properties.
More specifically, we shape the controllability Gramian \eqref{eq:unsta_Gram} by adjusting the feedback gain so that it becomes as close as possible to a desired symmetric positive definite matrix.
The next section presents an optimization method for shaping $\bm{P}$ into a desired form.

\section{Gramian Shaping with Bures--Wasserstein Distance}
This section explains the main idea of this paper, namely, Gramian shaping based on the BW distance.
Gramian shaping is a method for deforming the controllability Gramian into a desired form by exploiting the geometry of symmetric positive definite matrices \cite{Nishimoto2024}.
This enables the design of both ease and difficulty of control for each state when the system is operated externally.

In Gramian shaping, the input is designed so that the system approaches an ideal Gramian $\bm{P}_d$ representing the desired controllability.
The ideal Gramian may be defined from a reachable set based on input energy, or designed with reference to an ideal system.
Gramian shaping is achieved by minimizing the discrepancy between the controllability Gramian \eqref{eq:unsta_Gram} and the ideal Gramian $\bm{P}_d$.
In the previous study, this discrepancy between two Gramians was measured by the affine-invariant Riemannian (AIR) distance in accordance with the geometry of symmetric positive definite matrices \cite{Nishimoto2024}.
This made it possible to perform optimization while preserving positive definiteness.
However, optimization based on the AIR distance does not admit a guaranteed global optimum, and its physical interpretation is difficult.
To overcome these difficulties, this paper employs the BW distance, which is a different distance function. 
\begin{definition}[Bures--Wasserstein distance, \cite{BHATIA2019165}]\label{def:bw_distance}
    Given $\bm{X},\bm{Y}\in\mathbb{S}_{+}^n$ define $d_{\mathrm{BW}}(\bm{X},\bm{Y}):\mathbb{S}_+^n\times\mathbb{S}_+^n\to\mathbb{R}_+$ as Bures-Wasserstein distance by the relation
    \begin{align}
        d_{\mathrm{BW}}(\bm{X},\bm{Y})\!=\!\left(\mathrm{tr}\!\left(\bm{X}+\bm{Y}-2\left(\bm{Y}^{\frac{1}{2}}\bm{X}\bm{Y}^{\frac{1}{2}}\right)^{\frac{1}{2}}\right)\right)^{\frac{1}{2}}\!. 
        \label{eq:bw_distance}
    \end{align}
\end{definition}
The BW distance satisfies the axioms of a metric on $\mathbb{S}_{+}^n$ and $\mathbb{S}_{++}^n$.
It is known to coincide with the $2$-Wasserstein distance between Gaussian distributions having the same mean.

The BW distance enables us to measure the difference between two Gramians.
Our final goal is to design a feedback gain for Gramian shaping.
Using Lemmas~1 and~2, however, this problem can be separated into realizable Gramian design and gain derivation.
The first step is formulated as follows.
\begin{prob}[BW-based Gramian Shaping]\label{prob:Gramian_shaping_bw}
For the system \eqref{eq:outer_input_sys}, let $\bm{P}_d\in\mathbb{S}_{++}^n$ be a desired Gramian.
Under the BW distance \eqref{eq:bw_distance}, find an optimal realizable Gramian $\bm{P}^\ast$ by solving
\ifdefined\iflatexml
\begin{subequations}\label{eq:gramian_shaping_bw}
    \begin{align}
        \bm{P}^\ast=\underset{\bm{P}}{\arg\min}&~d_{\mathrm{BW}}^2(\bm{P},\bm{P}_d),\label{subeq:objective}\\
        \mathrm{s.t.}&~\bm{P}\in\mathbb{S}_{++}^n,\label{subeq:psd_const}\\
        &~\left(\bm{I}_n-\bm{BB}^\dagger \right)\left(\bm{A} \bm{P} + \bm{P}\bm{A}^\top \bm{DD}^\top\right)\left(\bm{I}_n-\bm{BB}^\dagger \right) = \bm{O}_n. \label{subeq:linear_const}
    \end{align}
\end{subequations}
\else
\begin{subequations}\label{eq:gramian_shaping_bw}
    \begin{align}
        \bm{P}^\ast=\underset{\bm{P}}{\arg\min}&~d_{\mathrm{BW}}^2(\bm{P},\bm{P}_d),\label{subeq:objective}\\
        \mathrm{s.t.}&~\bm{P}\in\mathbb{S}_{++}^n,\label{subeq:psd_const}\\
        &~\left(\bm{I}_n-\bm{BB}^\dagger \right)\left(\bm{A} \bm{P} + \bm{P}\bm{A}^\top \right.\nonumber\\
            &\left.+ \bm{DD}^\top\right)\left(\bm{I}_n-\bm{BB}^\dagger \right) = \bm{O}_n. \label{subeq:linear_const}
    \end{align}
\end{subequations}
\fi
\end{prob}
Constraint \eqref{subeq:psd_const} imposes positive definiteness of the Gramian, and constraint \eqref{subeq:linear_const} characterizes the set of Gramians realizable by the system (see Lemma~\ref{lem:gramian_theorem} in Appendix).
Thus, the controllability Gramian closest to the desired Gramian can be obtained.
The second step is to obtain a corresponding feedback gain $\bm{K}^\ast$ from $\bm{P}^\ast$ using Lemma~\ref{lem:gain_derive_for_gramian}.

We now prove that Problem~\ref{prob:Gramian_shaping_bw} is a strictly convex optimization problem.
We prove that the objective function of Problem~\ref{prob:Gramian_shaping_bw} is strictly convex.
\begin{theorem}\label{thm:bw_convex}
    For a fixed $\bm{P}_d\in\mathbb{S}_{++}^n$, $d_\mathrm{BW}^2(\bm{X},\bm{P}_d):\mathbb{S}_{++}^n\to\mathbb{R}_+$ is a strictly convex function on $\mathbb{S}_{++}^n$. 
\end{theorem}
\begin{proof}
Let $f(\bm{X})=d_\mathrm{BW}^2(\bm{X},\bm{P}_d)$. We verify that $f(\bm{X})$ satisfies the definition of a strictly convex function:
\begin{align}
    &tf(\bm{X})+(1-t)f(\bm{Y})> f(t\bm{X}+(1-t)\bm{Y}),\nonumber\\
    &~~~~~~~~~~~~~~~~~~~~~~~~~~~~~~~~~~~~~~~~~\forall t\in(0,1).
\end{align}
From Definition~\ref{def:bw_distance}, we have
\ifdefined\iflatexml
\begin{align}
    &tf(\bm{X})+(1-t)f(\bm{Y}) \nonumber\\=&\mathrm{tr}\left(t\bm{X}+(1-t)\bm{Y}+\bm{P}_d 
    -2t\left(\bm{P}_d^\frac{1}{2}\bm{X}\bm{P}_d^\frac{1}{2}\right)^\frac{1}{2}-2(1-t)\left(\bm{P}_d^\frac{1}{2}\bm{Y}\bm{P}_d^\frac{1}{2}\right)^\frac{1}{2}\right)\!,\\
    &f(t\bm{X}+(1-t)\bm{Y})\nonumber\\
    =&\mathrm{tr}\left(t\bm{X}+(1-t)\bm{Y}+\bm{P}_d 
    2\left(t\bm{P}_d^\frac{1}{2}\bm{X}\bm{P}_d^\frac{1}{2}+(1-t)\bm{P}_d^\frac{1}{2}\bm{Y}\bm{P}_d^\frac{1}{2}\right)^\frac{1}{2}\right).
\end{align}
\else
\begin{align}
    tf(\bm{X})&+(1-t)f(\bm{Y}) \nonumber\\=&\mathrm{tr}\Biggl(t\bm{X}+(1-t)\bm{Y}+\bm{P}_d \nonumber\\
    &-2t\left(\bm{P}_d^\frac{1}{2}\bm{X}\bm{P}_d^\frac{1}{2}\right)^\frac{1}{2}\left.\!\!\!-2(1-t)\left(\bm{P}_d^\frac{1}{2}\bm{Y}\bm{P}_d^\frac{1}{2}\right)^\frac{1}{2}\right)\!,\\
    f(t\bm{X}&+(1-t)\bm{Y})\nonumber\\
    =&\mathrm{tr}\Biggl(t\bm{X}+(1-t)\bm{Y}+\bm{P}_d \nonumber\\
    &-2\left(t\bm{P}_d^\frac{1}{2}\bm{X}\bm{P}_d^\frac{1}{2}\right.\left.\left.+(1-t)\bm{P}_d^\frac{1}{2}\bm{Y}\bm{P}_d^\frac{1}{2}\right)^\frac{1}{2}\right).
\end{align}
\fi
Subtracting the latter from the former yields
\begin{align}
    tf(\bm{X})+&(1-t)f(\bm{Y}) - f(t\bm{X}+(1-t)\bm{Y})\nonumber\\
    =&2\mathrm{tr}\!\!\left(\left(t\bm{P}_d^\frac{1}{2}\bm{X}\bm{P}_d^\frac{1}{2}+(1-t)\bm{P}_d^\frac{1}{2}\bm{Y}\bm{P}_d^\frac{1}{2}\right)^\frac{1}{2} \right) \nonumber\\
    -& 2\mathrm{tr}\!\!\left(t\!\left(\bm{P}_d^\frac{1}{2}\bm{X}\bm{P}_d^\frac{1}{2}\right)^\frac{1}{2}\!\!\!+\!(1-t)\!\left(\bm{P}_d^\frac{1}{2}\bm{Y}\bm{P}_d^\frac{1}{2}\right)^\frac{1}{2}\!\right)\!.
\end{align}
Let $\bm{X}^\prime = \bm{P}_d^\frac{1}{2}\bm{X}\bm{P}_d^\frac{1}{2}$ and $\bm{Y}^\prime = \bm{P}_d^\frac{1}{2}\bm{Y}\bm{P}_d^\frac{1}{2}$, then this equation can be rewritten as
\begin{align}\label{eq:totyuushiki}
    tf(\bm{X})+&(1-t)f(\bm{Y}) - f(t\bm{X}+(1-t)\bm{Y})\nonumber\\
    =&2\mathrm{tr}\left(\left(t\bm{X}^\prime+(1-t)\bm{Y}^\prime\right)^\frac{1}{2} \right) \nonumber\\
    &- 2\left(t\mathrm{tr}\left(\bm{X}^{\prime\frac{1}{2}}\right)+(1-t)\mathrm{tr}\left(\bm{Y}^{\prime\frac{1}{2}}\right)\right).
\end{align}
We now examine the sign of the left-hand side. 
From Lemma~\ref{lem:concave} in Appendix, this function $\mathrm{tr}(\bm{X}^{\frac{1}{2}})$ is strictly concave on $\mathbb{S}_{++}^n$. That is,
\begin{align}
    &\mathrm{tr}\left(\left(t\bm{X}^\prime+(1-t)\bm{Y}^\prime\right)^\frac{1}{2} \right) \nonumber\\
    &- \left(t\mathrm{tr}\left(\bm{X}^{\prime\frac{1}{2}}\right)+(1-t)\mathrm{tr}\left(\bm{Y}^{\prime\frac{1}{2}}\right)\right)> 0,\nonumber\\
    &~~~~~~~~~~~~~~~~~~~~~~~~~~~~~~~~~~~~\forall t\in(0,1).
\end{align}
Since the right-hand side of \eqref{eq:totyuushiki} is strictly positive, it follows that
\begin{align}
    &tf(\bm{X})+(1-t)f(\bm{Y}) - f(t\bm{X}+(1-t)\bm{Y})> 0,\nonumber\\
    &~~~~~~~~~~~~~~~~~~~~~~~~~~~~~~~~~~~~~~~~~~~~~\forall t\in(0,1).
\end{align}
We therefore conclude that $f(\bm{X})$ is strictly convex on $\mathbb{S}_{++}^n$.
\end{proof}
Using Theorem~\ref{thm:bw_convex}, we can show that Problem~\ref{prob:Gramian_shaping_bw} is a strictly convex optimization problem.
\begin{theorem}\label{thm:main_convexity}
Problem~\eqref{eq:gramian_shaping_bw} is a convex optimization problem with a strictly convex objective function. Hence, if an optimal solution exists, it is unique.
\end{theorem}
\begin{proof}
The feasible set is convex because $\mathbb{S}_{++}^n$ is convex and the constraint \eqref{subeq:linear_const} is affine in $\bm{P}$.
By Theorem~\ref{thm:bw_convex}, the objective function $d_{\mathrm{BW}}^2(\bm{P},\bm{P}_d)$ is strictly convex on $\mathbb{S}_{++}^n$.
Therefore, Problem~\ref{prob:Gramian_shaping_bw} is a convex optimization problem with a strictly convex objective function.
Hence, if an optimal solution exists, it is unique.
\end{proof}

In the AIR-based approach \cite{Nishimoto2024}, the optimization problem is nonconvex.
This may cause dependence on the initial value, convergence to a local minimum, and make the calculation inefficient.
By contrast, when the BW distance is used, the problem becomes convex, and any converged solution of the solver is guaranteed to be globally optimal.
This provides Gramian shaping with the robustness of the optimization problem for which the mathematically best solution is guaranteed.

\section{Solving Gramian Shaping \\as Semidefinite Programming}
This section explains how to solve the Gramian shaping problem numerically.
As a naive approach, one can first solve Problem~\ref{prob:Gramian_shaping_bw} directly and then determine the gain via Lemma~\ref{lem:gain_derive_for_gramian}.
On the other hand, by exploiting the properties of the BW distance, we can obtain a more efficient optimization formulation.

We first reconsider Problem~\ref{prob:Gramian_shaping_bw} in a form that explicitly includes gain determination.
\begin{cor}
    The solution $\bm{P}^\ast\in\mathbb{S}_{++}^n$ to the optimization problem \eqref{eq:gramian_shaping_bw} and the corresponding gain $\bm{K}$ can be obtained by solving the optimization problem:
    \begin{subequations}\label{eq:bw_gramain_gain_problem}
    \begin{align}
      \min_{\bm{P},\bm{K}}&~d_\mathrm{BW}^2(\bm{P},\bm{P}_d)\label{subeq:objective_neo},\\
      \mathrm{s.t.}&~\bm{P}\in\mathbb{S}_{++}^n,\label{subeq:SPD_const}\\
      &~(\bm{A}\!+\!\bm{B}\bm{K})\bm{P} \!+\! \bm{P}(\bm{A}\!+\!\bm{B}\bm{K})^\top \!+\! \bm{D}\bm{D}^\top\!=\!\bm{O}_n.\label{subeq:lyapunov_const}
    \end{align}
  \end{subequations}
\end{cor}
\begin{proof}
    This follows immediately from Lemmas~\ref{lem:gramian_theorem} and Lemma~\ref{lem:gain_derive_for_gramian} in Appendix.
\end{proof}

Using this LMI representation of the BW distance, we derive an SDP formulation of the Gramian shaping problem.
\begin{prob}[BW-based Gramian Shaping with LMIs]\label{prob:sdp_gramian_shaping}
For the system \eqref{eq:outer_input_sys}, let $\bm{P}_d\in\mathbb{S}_{++}^n$ be a desired Gramian.
Using an LMI representation of the BW distance \eqref{eq:bw_distance}, find an optimal realizable Gramian $\bm{P}^\ast$ by solving
\begin{subequations}\label{eq:bw_gramain_gain_problem_all}
    \begin{align}
      (\bm{P}^\ast,\bm{E}^\ast,\bm{U}^\ast)=\underset{\bm{P},\bm{E},\bm{U}}{\arg\min}&~\mathrm{tr}\left(\bm{P}+\bm{P}_d-2\bm{U}\right),\\
      \mathrm{s.t.}&~\bm{P}\in\mathbb{S}_{++}^n,\\
      &~\begin{bmatrix}
          \bm{P}&\bm{U}\\
          \bm{U}^\top&\bm{P}_d
        \end{bmatrix}\succeq 0,\\
      &~\bm{A}\bm{P}\!+\!\bm{B}\bm{E}\!+\!\bm{P}\bm{A}^\top \!+\!\bm{E}^\top\bm{B}^\top \nonumber\\
      &~~~~\!+\! \bm{D}\bm{D}^\top\!=\!\bm{O}_n.
    \end{align}
\end{subequations}
\end{prob}
A corresponding feedback gain that achieves the optimal realizable Gramian is obtained as
\begin{align}
    \bm{K}^\ast=\bm{E}^\ast\bm{P}^{\ast^{-1}}.
\end{align}
Problem~\ref{prob:sdp_gramian_shaping} can be solved numerically as an SDP.
In practice, the condition $\bm{P}\in\mathbb{S}_{++}^n$ is typically imposed as $\bm{P}\succeq\epsilon\bm{I}$ with a sufficiently small positive scalar $\epsilon$.

\begin{theorem}
    The solution to the optimization problem \eqref{eq:gramian_shaping_bw} can be obtained by solving the problem~\eqref{eq:bw_gramain_gain_problem_all}. 
\end{theorem}
\begin{proof}
Since the optimal solution $\bm{P}^\ast$ of \eqref{eq:gramian_shaping_bw} coincides with that of \eqref{eq:bw_gramain_gain_problem}, it suffices to prove the equivalence between \eqref{eq:bw_gramain_gain_problem} and \eqref{eq:bw_gramain_gain_problem_all}.
First, because $\bm{P}\in\mathbb{S}_{++}^n$, $\bm{P}^{-1}$ exists.
Hence, under the change of variables $\bm{E}=\bm{K}\bm{P}$, the Lyapunov constraint \eqref{subeq:lyapunov_const} is equivalently rewritten as
\begin{align}
    \bm{A}\bm{P}+\bm{B}\bm{E}+\bm{P}\bm{A}^\top+\bm{E}^\top\bm{B}^\top+\bm{D}\bm{D}^\top=\bm{O}_n.
\end{align}
Conversely, for any $(\bm{P},\bm{E})$, setting $\bm{K}=\bm{E}\bm{P}^{-1}$ recovers the original constraint.
Therefore, the formulations using $(\bm{P},\bm{K})$ and $(\bm{P},\bm{E})$ are equivalent.
Moreover, by Lemma~\ref{lem:bures_min} in Appendix, the optimization problem \eqref{eq:bw_gramain_gain_problem} can be rewritten as
\begin{subequations}\label{eq:bw_gramain_gain_problem_partial}
    \begin{align}
      \min_{\bm{P},\bm{E}}&~\left(\begin{array}{c}
          \underset{\bm{U}}{\min}~\mathrm{tr}\left(\bm{P}+\bm{P}_d-2\bm{U}\right),\\
          \mathrm{s.t.}\begin{bmatrix}
          \bm{P}&\bm{U}\\
          \bm{U}^\top&\bm{P}_d
        \end{bmatrix}\succeq 0
      \end{array}\right)\\
      \mathrm{s.t.}&~\bm{P}\in\mathbb{S}_{++}^n,\\
            &~\bm{A}\bm{P}\!+\!\bm{B}\bm{E} \!+\!\bm{P}\bm{A}^\top\!+\!\bm{E}^\top\bm{B}^\top \!+\! \bm{D}\bm{D}^\top\!=\!\bm{O}_n.
    \end{align}
\end{subequations}
By Lemma~\ref{lem:minmin} in Appendix, partial minimization with respect to distinct optimization variables is equivalent to simultaneous optimization.
Hence, \eqref{eq:bw_gramain_gain_problem_partial} and \eqref{eq:bw_gramain_gain_problem_all} are equivalent, and the optimal solutions of \eqref{eq:gramian_shaping_bw} and \eqref{eq:bw_gramain_gain_problem_all} coincide.

\end{proof}

We conclude that the Gramian shaping problem can be solved as an SDP.
Since many fast SDP solvers are available, the problem is numerically tractable.
Moreover, because the problem admits an SDP formulation, additional LMI constraints can be incorporated into the Gramian design.
Since many control constraints, such as upper bounds on Gramian entries and norm bounds on the internal control input \cite{LMI_in_system_and_control}, can be expressed as LMIs, this SDP-based Gramian shaping significantly improves the extensibility of the design framework.

\section{Connection between Gramian Shaping \\and $H_2$ control}\label{sec:H2_control}
This section discusses the relation between the proposed Gramian shaping framework and existing control theory.
In particular, we focus on its connection to $H_2$ control.

In this section, the exogenous input $\bm v\in\mathbb{R}^r$ in \eqref{eq:stab_syst} is assumed to be zero-mean Gaussian white noise.
That is,
\begin{align}
    \mathbb{E}[\bm v(t)] &= \bm 0_r,\qquad
    \mathbb{E}[\bm v(t)\bm v(\tau)^\top] = \bm I_r\,\delta(t-\tau)
\end{align}
where $\delta(t):\mathbb{R}\to\mathbb{R}$ denotes the Dirac delta function.
Then the state $\bm x(t)$ becomes a stochastic process, and if the closed-loop matrix $\bm A+\bm B\bm K$ is Hurwitz stable, the stationary covariance
\begin{align}
    \bm P_w := \lim_{t\to\infty}\mathbb{E}\!\left[\bm x(t)\bm x(t)^\top\right]
\end{align}
exists.
Moreover, $\bm P_w$ is the unique positive-semidefinite solution of the Lyapunov equation
\begin{align}
    (\bm A+\bm B\bm K)\bm P_w + \bm P_w(\bm A+\bm B\bm K)^\top + \bm D\bm D^\top = \bm O_n
    \label{eq:cov_lyap}
\end{align}
and is expressed as
\begin{align}
    \bm P_w
    = \int_{0}^{\infty} e^{\bm (\bm A+\bm B\bm K) \tau}\,\bm D\bm D^\top\,e^{\bm (\bm A+\bm B\bm K)^\top \tau}\,d\tau
    \label{eq:cov_gramian_int}
\end{align}
Therefore, when $\bm{v}$ is Gaussian white noise, the state covariance coincides mathematically with the controllability Gramian in Definition~\ref{def:unsta_Gram}:
\begin{align}
    \bm P_w=\bm{P}.
\end{align}

We now consider $H_2$ control for the system \eqref{eq:outer_input_sys}.
$H_2$ control is based on the $H_2$ norm of the transfer function $\bm{G}(s)$ from the disturbance $\bm{v}$ to the state $\bm{x}$.
\begin{align}
    \bm{X}(s)&=\bm{G}(s)\bm{W}(s)\\
    \bm{G}(s)&=\left(s\bm{I}_n-\bm{A}-\bm{B}\bm{K}\right)^{-1}\bm{D}\\
    \left\|\bm{G}\right\|_2 &= \sqrt{\frac{1}{2\pi}\int_{-\infty}^{\infty}\mathrm{tr}\left(\bm{G}^H(jw)\bm{G}(jw)\right)dw}\label{eq:H2_norm}
\end{align}
$H_2$ control is defined as the optimization problem of finding the gain $\bm{K}$ that minimizes the $H_2$ norm of $\bm{G}(s)$.
\begin{subequations}\label{eq:H2_origin}
    \begin{align}
        \min_{\bm K}~&\left\|\bm G\right\|_2^2\\
        \mathrm{s.t.}~&\bm{G}=\left[\begin{array}{c|c}
        \bm{A}+\bm{B}\bm{K}&\bm{D}\\\hline
        \bm{I}_n&\bm{O}_{n\times r}
        \end{array}\right]
    \end{align}
\end{subequations}

The $H_2$ control \eqref{eq:H2_origin} can be solved as an SDP.
To this end, the $H_2$ norm is converted into a time-domain expression via Parseval's theorem.
\begin{prop}[Sec.~4.10~\cite{Multivariable_Feedback_Control_Analysis_and_Design}]
    For the system \eqref{eq:outer_input_sys}, the $H_2$ norm \eqref{eq:H2_norm} of the transfer function $\bm{G}$ from the disturbance $\bm{v}$ to the state $\bm{x}$ is given by
    \begin{align}
        \left\|\bm{G}\right\|_2 = \sqrt{\mathrm{tr}\left(\bm{P}_w\right)}.
    \end{align}
\end{prop}

Hence, $H_2$ control can be written as the following optimization problem.
\begin{prob}[$H_2$ control with LMIs]
\label{prob:H2_control}
For the system \eqref{eq:outer_input_sys}, consider a problem of designing a feedback gain that minimizes the $H_2$ norm of the transfer function from the disturbance $\bm{v}$ to the state $\bm{x}$:
    \begin{subequations}\label{eq:H2_gram}
        \begin{align}
            (\bm{P}_w^\ast,\bm{E}^\ast)=\underset{\bm{P}_w,\bm{E}}{\arg\min}~&\mathrm{tr}\left(\bm{P}_w\right),\\
            \mathrm{s.t.}~&\bm{P}_w\in\mathbb{S}_{++}^n,\\
            &\bm{A}\bm{P}_w+\bm{B}\bm{E}+\bm{P}_w\bm{A}^\top\nonumber\\
            &~~~+\bm{E}^\top\bm{B}^\top+\bm{D}\bm{D}^\top=\bm{O}_n.
        \end{align}
    \end{subequations}
    The corresponding feedback gain is recovered as $\bm{K}^\ast=\bm{E}^\ast\bm{P}_w^{{\ast}^{-1}}$.
\end{prob}

Note that also in Problem~\ref{prob:H2_control}, in order to compute $\bm{K}^\ast=\bm{E}^\ast\bm{P}_w^{\ast^{-1}}$ stably in numerical implementation, one may impose the constraint $\bm{P}_w\succeq \epsilon \bm{I}_n$.
This is a numerical regularization and should be distinguished from the theoretical problem formulation.

The correspondence between the $H_2$ control problem in Problem~\ref{prob:H2_control} and the Gramian shaping problem in Problem~\ref{prob:sdp_gramian_shaping} can be verified directly from the closed-form expression \eqref{eq:bw_distance} of the BW distance.
If we set $\bm P_d=\alpha\bm I_n$ and let $\alpha\to 0$, the cross term vanishes and
\begin{align}
\lim_{\alpha\to0} d_{\mathrm{BW}}^2(\bm P,\alpha\bm I_n)=\mathrm{tr}(\bm P)
\end{align}
holds.
In this case, the objective function of Problem~\ref{prob:sdp_gramian_shaping} reduces to $\mathrm{tr}(\bm P)$, and the constraints are identical.
Hence, it coincides with the $H_2$ state-feedback design problem in Problem~\ref{prob:H2_control}.
Therefore, in this limit, BW-based Gramian shaping can be regarded as a generalization that contains $H_2$ control as a special case.
It can be interpreted as a framework that generalizes minimization of the scalar measure $\mathrm{tr}(\bm P_w)$, which measures the magnitude of the covariance (or controllability Gramian), to minimization of the distance between the covariance matrix and a target matrix.
Furthermore, from a probabilistic viewpoint, under Gaussian white disturbance, the stationary distribution of the closed-loop system \eqref{eq:stab_syst} is $\mathcal N(\bm 0,\bm P_w)$,
and the BW distance coincides with the $2$-Wasserstein distance between zero-mean Gaussian distributions.
Thus, $H_2$ control can be reinterpreted as a design that drives the stationary closed-loop distribution $\mathcal N(\bm 0,\bm P_w)$ toward the degenerate Gaussian distribution $\mathcal N(\bm 0,\bm O_n)$ centered at the origin in the sense of optimal transport.
This viewpoint reveals, through Gramian shaping, that classical $H_2$ optimal control implicitly performs an optimal transport of state distributions.

\section{Examples}
This section verifies the proposed method through numerical examples.

\subsection{Gramian shaping for a guidance robot}

A guidance robot \cite{Kayukawa2022,Takagi2025} is required to follow a reference path while reflecting user inputs, as shown in Fig.~\ref{fig:trajectory_tracking_of_AIsuitcase}.
Here, Gramian shaping is used to increase controllability degree along the path and suppress it in the perpendicular direction.
We also examine the design flexibility provided by additional LMI constraints.

The robot is modeled by the following acceleration-input unicycle model:
\begin{align}
    \dot{\bm{x}} =
    \begin{bmatrix}
        c\cos(\theta)\\
        c\sin(\theta)\\
        \omega\\
        0\\
        0
    \end{bmatrix}
    +
    \begin{bmatrix}
        0&0\\
        0&0\\
        0&0\\
        1&0\\
        0&1
    \end{bmatrix}
    (\bm{u}+\bm{v}),
\end{align}
where $\bm{u}\in\mathbb{R}^2$ is the internal input and $\bm{v}\in\mathbb{R}^2$ is the user-applied exogenous input.
As shown in Fig.~\ref{fig:trajectory_tracking_of_AIsuitcase}, forward pulling affects translational motion, whereas lateral pulling induces rotation. Therefore, the internal input and the exogenous input are assumed to enter the system through the same input matrix.

Let $\bm{x}^\ast=\bm{p}(s^\ast)$ be the closest point on the reference path $\bm{p}(s)$, and define the state deviation by
$\delta\bm{x}=\bm{x}-\bm{x}^\ast=\left[\delta p_x~  \delta p_y~\delta\theta~  \delta c~  \delta\omega\right]^\top$.
Next, introduce the path-parallel and path-perpendicular errors $e_{\parallel}$ and $e_{\perp}$ by the coordinate transformation
\begin{align}
\begin{bmatrix}
    e_{\parallel}\\
    e_{\perp}
\end{bmatrix}
=
\bm{R}^\top(\theta^\ast)
\begin{bmatrix}
    \delta p_x\\
    \delta p_y
\end{bmatrix},
\end{align}
and define
    $\delta \bm{\xi}=
    \left[
        e_{\parallel} ~ e_{\perp} ~ \delta\theta ~ \delta c ~ \delta\omega
    \right]^\top$.
Then, linearizing the system around $\bm{x}^\ast$ yields
\begin{align}
    \delta\dot{\bm{\xi}}=\bm{A}\delta\bm{\xi}+\bm{B}\delta\bm{u}+\bm{B}\bm{v},
\end{align}
with
\begin{align}
    \bm{A}=
    \begin{bmatrix}
        0 & \omega^\ast & 0 & 1 & 0\\
        -\omega^\ast & 0 & c^\ast & 0 & 0\\
        0 & 0 & 0 & 0 & 1\\
        0 & 0 & 0 & 0 & 0\\
        0 & 0 & 0 & 0 & 0
    \end{bmatrix},
    \quad
    \bm{B}=
    \begin{bmatrix}
        0 & 0\\
        0 & 0\\
        0 & 0\\
        1 & 0\\
        0 & 1
    \end{bmatrix},
\end{align}
where $\delta\bm{u}=\bm{u}-\bm{u}^\ast$ and $\bm{u}^\ast$ is the nominal input achieving the reference state.

In the simulation, the user pulls the robot forward, and thus $\bm{v}=[1~0]^\top$.
The reference trajectory consists of straight and circular segments.
The desired Gramian is set to
\begin{align}
    \bm{P}_d=\mathrm{diag}(1,0.001,1,0.1,1),
\end{align}
so that controllability is preserved along the path while being suppressed in the perpendicular direction.

We compare the following four cases.
Case 1 has no exogenous input.
Case 2 applies AIR-based Gramian shaping with exogenous input.
Case 3 applies BW-based Gramian shaping with exogenous input \eqref{eq:bw_gramain_gain_problem_all}.
Case 4 adds the constraint
\begin{align}\label{eq:added_lmi_const}
    \bm{P}(2,2)\leq 0.001
\end{align}
to the BW-based Gramian shaping problem \eqref{eq:bw_gramain_gain_problem_all} in order to suppress the path-perpendicular error more strictly.

The resulting Gramians for Cases 2 and 3 are
{\setlength{\arraycolsep}{1pt}
\begin{flalign}
    \bm{P}^\ast_{2}=
    \begin{bmatrix}
        1&0&0&0&0\\
        0&0.0648&0&0&-0.0722\\
        0&0&0.7224&0&0\\
        0&0&0&0.1000&0\\
        0&-0.0722&0&0&1.0778
    \end{bmatrix},
\end{flalign}
}
{\setlength{\arraycolsep}{1pt}
\begin{flalign}
    &\bm{P}^\ast_{3}=
    \begin{bmatrix}
        0.9993&0&0&0&0\\
        0&0.0020&0&0&-0.0401\\
        0&0&0.4008&0&0\\
        0&0&0&0.1001&0\\
        0&-0.0401&0&0&2.0258
    \end{bmatrix},&
\end{flalign}
}
respectively.
In both cases, the perpendicular component exceeds the target value $0.001$, so the desired Gramian is not exactly achieved.
By contrast, directly imposing \eqref{eq:added_lmi_const} in Case 4 yields
{\setlength{\arraycolsep}{1pt}
\begin{flalign}
    \bm{P}^\ast_{4}=
    \begin{bmatrix}
        1&0&0&0&0\\
        0&0.0010&0&0&-0.0216\\
        0&0&0.2161&0&0\\
        0&0&0&0.1000&0\\
        0&-0.0216&0&0&1.2447
    \end{bmatrix}.
\end{flalign}
}
This illustrates an advantage of the BW-based formulation: additional design constraints can be incorporated directly as LMIs in the SDP.

The control period is $0.2$~seconds, and the input is applied with zero-order hold.
Each simulation ends when the robot reaches a ball of radius $0.1$ centered at $({p}_{x,\mathrm{goal}},{p}_{y,\mathrm{goal}})=(-2,2)$.
We also evaluate the average computation time of Gramian shaping at each step.
The simulations are carried out in MATLAB 2024b.
For AIR-based Gramian shaping, \texttt{fmincon} is used, whereas BW-based Gramian shaping is solved by YALMIP.
The computer environment is an Intel Core i7-12700H CPU, Windows 11 Pro, and 32 GB memory.

Table~\ref{tab:goal_times} lists the goal-reaching times and average computation times.
Cases 2--4 all reach the goal faster than Case 1, showing that the forward user input is properly reflected.
Moreover, Cases 3 and 4 are much faster to compute than Case 2, confirming the computational advantage of BW-based Gramian shaping.

Fig.~\ref{fig:path_following} shows the trajectories and the path-perpendicular error $\|e_{\perp}\|$.
In Cases 2 and 3, the perpendicular controllability is not sufficiently suppressed, so the exogenous input also affects the lateral direction and causes deviation from the path.
In Case 4, this deviation is reduced without significantly degrading the goal-reaching time.
This is because the LMI constraint directly limits the influence in the perpendicular direction.

\begin{figure}[t]
  \begin{center}
    \includegraphics[width=60mm]{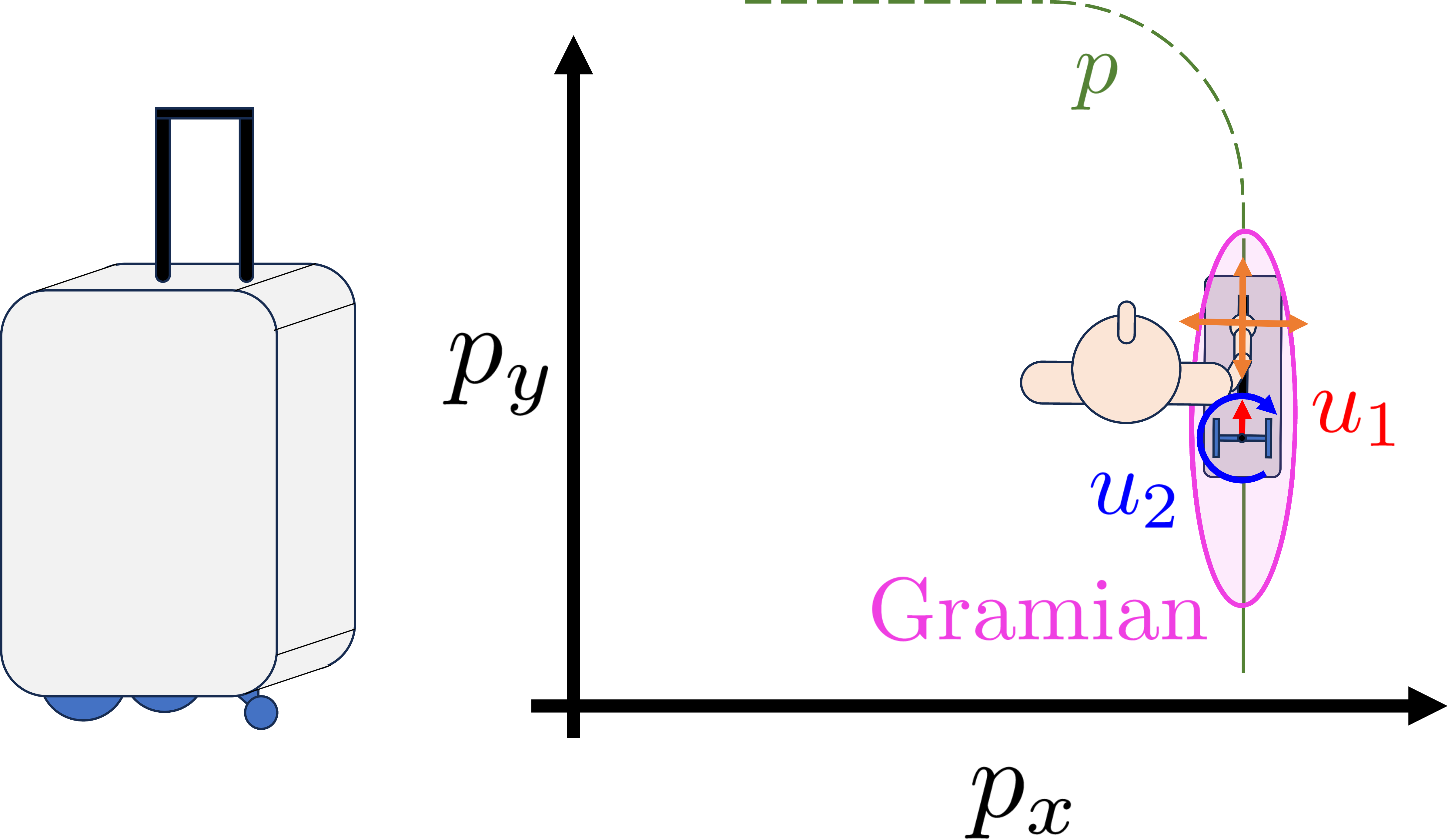}
    \caption{
    Visualization of a guidance robot pulling a user while following a track. 
    The user can give the robot some instructions with the handle. 
    }
    \label{fig:trajectory_tracking_of_AIsuitcase}
  \end{center}
\end{figure}

\begin{figure}[t]
    \centering
    \includegraphics[width=80mm]{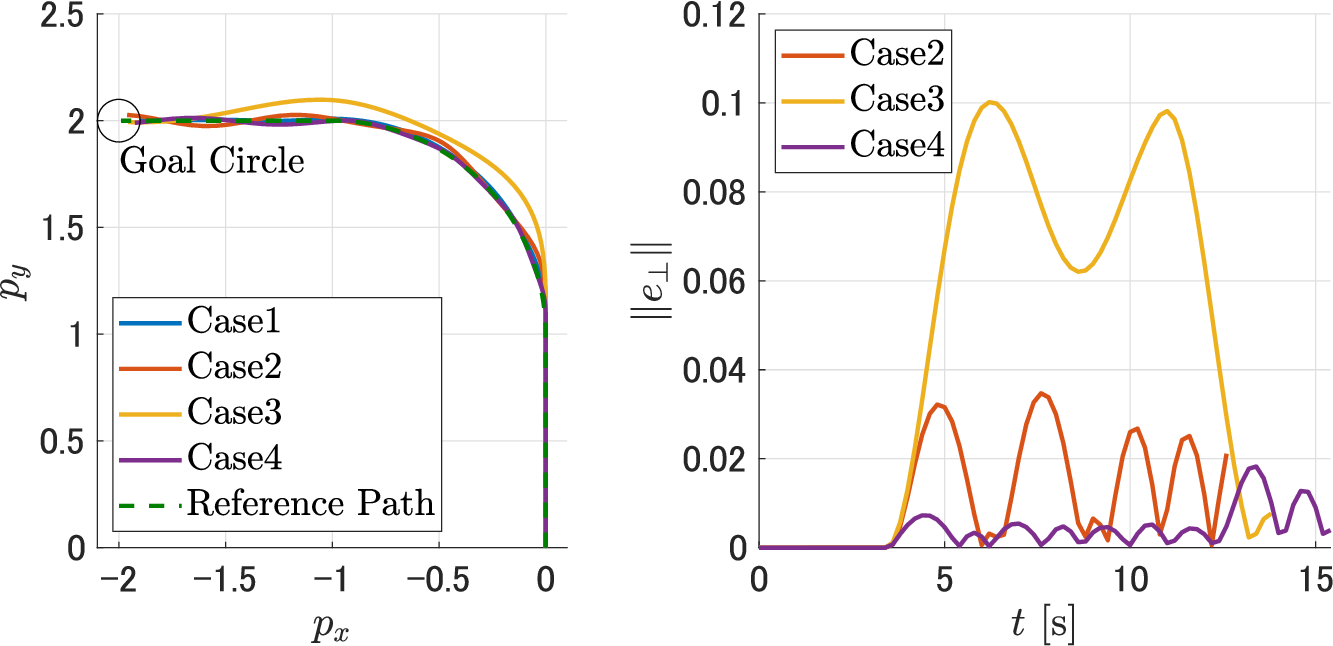}
    \caption{Path-following trajectories in Cases 1--4 (Left) and path-perpendicular errors in Cases 2--4 (Right).
    In Cases 2 and 3, the exogenous input is reflected but causes deviation from the reference path.
    In Case 4, the additional LMI constraint suppresses this deviation.}
    \label{fig:path_following}
\end{figure}

\begin{table}[t]
  \centering  
  \caption{Goal-reaching times and average computation times \\for Cases 1--4}
  \label{tab:goal_times}
  \begin{tabular}{|l|c|c|c|c|}\hline
    &Case1 & Case2 & Case3 & Case4 \\\hline
    Goal Time[s]&35.8 & 12.6 & 13.8 & 15.4 \\\hline
    Average computation time[s]&$-$&0.328 &0.100&0.086\\\hline
  \end{tabular}
\end{table}

\subsection{Numerical comparison between $H_2$ control and Gramian shaping}

Section~\ref{sec:H2_control} showed that BW-based Gramian shaping with $\bm{P}_d=\alpha \bm{I}_n$ converges to $H_2$ control as $\alpha\to 0$.
This section verifies the relationship by a simple two-dimensional example.

Consider the system \eqref{eq:outer_input_sys} with
\begin{align}
    \bm{A}=\begin{bmatrix}
        0&1\\
        -1&-1
    \end{bmatrix},\quad
    \bm{B}=\begin{bmatrix}
        0\\
        1
    \end{bmatrix},\quad
    \bm{D}=\begin{bmatrix}
        1\\
        0
    \end{bmatrix}.
\end{align}
The exogenous input $\bm{v}$ is white noise. 

We compare the stationary covariances, equivalently the controllability Gramian, obtained by $H_2$ control with that obtained by BW-based Gramian shaping.
In BW-based Gramian shaping, the desired Gramian is set to $\bm{P}_d=\alpha \bm{I}_n$, and the resulting closed-loop covariance is compared with that of $H_2$ control for different values of $\alpha$.
To avoid excessively high gains, the positive definite constraint $\bm{P}\succeq 10^{-2}\bm{I}_n$ is imposed.

The results are shown in Figs.~\ref{fig:covariance_spl} and \ref{fig:compare_gram_H2}.
Fig.~\ref{fig:covariance_spl} plots 1000 state samples at $t=100$ from the initial condition $\bm{x}(0)=\bm{0}$ for each $\alpha$ and for $H_2$ control.
The trajectories are simulated by the Euler--Maruyama method.
For large $\alpha$, the sample distribution spreads according to the target Gramian $\bm{P}_d=\alpha \bm{I}_n$.
As $\alpha$ decreases, it approaches the distribution obtained by $H_2$ control.

Fig.~\ref{fig:compare_gram_H2} shows the Frobenius norm of the difference between the stationary covariance for BW-based Gramian shaping and that for $H_2$ control, namely,
\begin{align}
    \|\bm{P}^\ast-\bm{P}_{H_2}\|_F.
\end{align}
This confirms that the BW-based solution continuously approaches the $H_2$ solution as $\alpha$ decreases, and coincides with it at $\alpha=0$.

These results numerically support that BW-based Gramian shaping reduces to $H_2$ control in the limit $\bm{P}_d\to\bm{O}_n$.
Hence, when BW-based Gramian shaping is interpreted as distribution control for linear stochastic systems, $H_2$ control can be viewed as a transport problem toward a degenerate Gaussian distribution with zero covariance.

\begin{figure}[t]
    \centering
    \includegraphics[width=80mm]{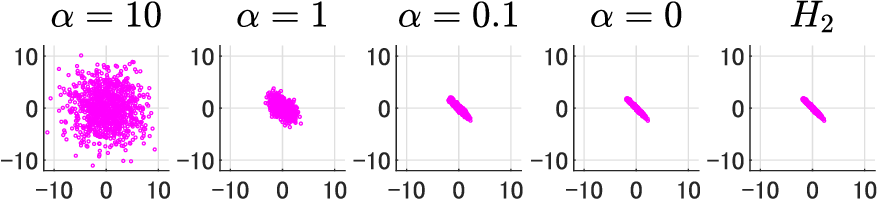}
    \caption{State sample distributions under BW-based Gramian shaping for different $\alpha$ and under $H_2$ control.
    As $\alpha$ decreases, the distribution approaches that of $H_2$ control.}
    \label{fig:covariance_spl}
\end{figure}

\begin{figure}[t]
    \centering
    \includegraphics[width=50mm]{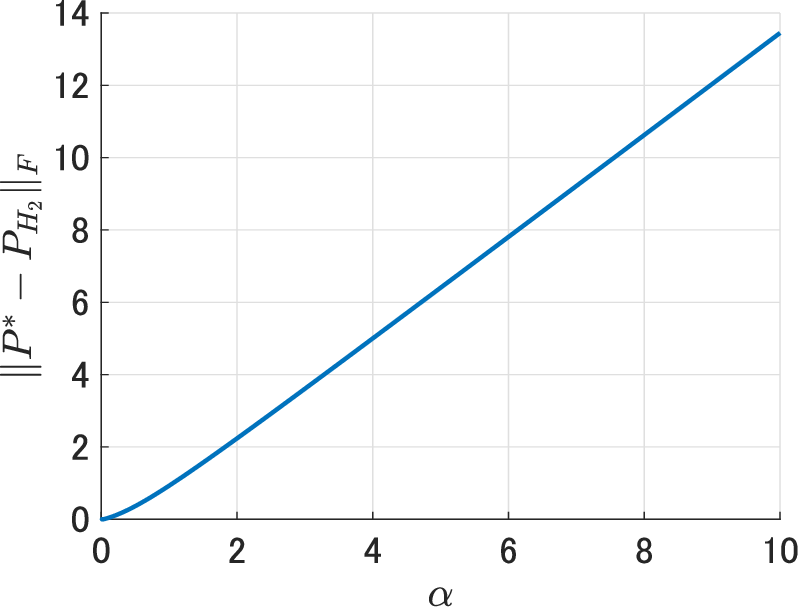}
    \caption{Difference $\|\bm{P}^\ast-\bm{P}_{H_2}\|_F$ between the stationary covariance obtained by BW-based Gramian shaping and that obtained by $H_2$ control.
    The step size of $\alpha$ is 0.01.
    The difference decreases as $\alpha$ becomes smaller.}
    \label{fig:compare_gram_H2}
\end{figure}

\section{CONCLUSIONS}

This paper proposed a BW-based method for shaping the controllability Gramian from exogenous inputs to the state into a desired form.
We showed that the objective function is strictly convex and formulated the problem as an SDP using an LMI representation of the BW distance. This also enables the incorporation of additional LMI constraints.
The proposed framework also clarified that $H_2$ control is a special case of BW-based Gramian shaping under Gaussian white noise.
Numerical examples demonstrated anisotropic controllability design for a guidance robot, verified the ability to impose additional LMI constraints, and confirmed through a simple example that the proposed method approaches $H_2$ control in the corresponding limit.
These results show the effectiveness of the proposed method for externally operated systems.

\newpage
\appendix
This appendix contains the lemmas used in this paper.

\begin{lem}[Th.~4~\cite{Hotz1985}]
\label{lem:gramian_theorem}
    For an ($\bm{A},\bm{B}$)-controllable and ($\bm{A},\bm{D}$)-controllable linear system \eqref{eq:outer_input_sys}, there exists a controllability Gramian $\bm{P}$ if and only if
    \ifdefined\iflatexml
    \begin{align}\label{eq:condition_of_gain_gram}
        &\left(\bm{I}_n-\bm{BB}^\dagger \right)\left(\bm{A} \bm{P} + \bm{P}\bm{A}^\top + \bm{DD}^\top\right)\left(\bm{I}_n-\bm{BB}^\dagger \right) = \bm{O}_n.
    \end{align}
    \else
    \begin{align}\label{eq:condition_of_gain_gram}
        &\left(\bm{I}_n-\bm{BB}^\dagger \right)\left(\bm{A} \bm{P} + \bm{P}\bm{A}^\top \right.\nonumber\\
        &~~~~~~~~~~~~~~~~~~~\left.+ \bm{DD}^\top\right)\left(\bm{I}_n-\bm{BB}^\dagger \right) = \bm{O}_n.
    \end{align}
    \fi
\end{lem}
\begin{lem}[Th.~5~\cite{Hotz1985}]
\label{lem:gain_derive_for_gramian}
For the controllable linear system \eqref{eq:outer_input_sys}, assume that a controllability Gramian $\bm{P}$ is realizable by a stabilizing linear feedback law.
Then all gain matrices $\bm{K}$ realizing $\bm{P}$ are given by
\begin{align} \bm{K} = -\frac{1}{2}\bm{B}^\dagger \left(\bm{A} \bm{P} + \bm{P} \bm{A}^\top +{\bm{D}\bm{D}^\top}  - \bm{S}_d \right) \bm{P}^{-1}, \label{eq:specific_gain} \end{align}
where $\bm{S}_d$ is an $n \times n$ skew-symmetric matrix of the form
\begin{align} 
\bm{S}_d = \bm{N}\left[\begin{array}{cc} \bm{O}_{n-m} & \bm{S}_{d_{12}} \\ -\bm{S}_{d_{12}}^\top  & \bm{S}_{d_{22}} \end{array}\right]\bm{N}^\top.
\end{align}
$\bm{N}\in\mathcal{O}(n)$ satisfies 
\begin{align}
    \bm{N}\left(\bm{I}-\bm{B}\bm{B}^\dagger\right)\bm{N}^\top=\mathrm{blkdiag}\left(\bm{I}_{n-m}, \bm{O}_m\right).
\end{align}
$\bm{S}_{d12}$ is given by
\ifdefined\iflatexml
\begin{align}
    \bm{S}_{d12}=&\left[\bm{I}_{n-m}~\bm{O}_{n-m\times m}\right]\bm{N}^\top\bigl(\bm{A}\bm{P}+\bm{P}\bm{A}^\top+{\bm{D}\bm{D}^\top}\bigr)\bm{N}\left[\bm{O}_{n-m\times m}~\bm{I}_{n-m}\right]^\top
\end{align}
\else
\begin{align}
    \bm{S}_{d12}=&\left[\bm{I}_{n-m}~\bm{O}_{n-m\times m}\right]\bm{N}^\top\bigl(\bm{A}\bm{P}+\bm{P}\bm{A}^\top\nonumber\\
    &+{\bm{D}\bm{D}^\top}\bigr)\bm{N}\left[\bm{O}_{n-m\times m}~\bm{I}_{n-m}\right]^\top
\end{align}
\fi
and $\bm{S}_{d22}$ is an arbitrary $m\times m$ skew-symmetric matrix.
\end{lem}

\begin{lem}[Th.~7~\cite{BHATIA2019165}]\label{lem:concave}
    $\mathrm{tr}(\bm{X}^{\frac{1}{2}}):\mathbb{S}_+^n\to\mathbb{R}_+$ is a strictly concave function on $\mathbb{S}_{++}^n$. That is,  $\forall\bm{X},\bm{Y}\in\mathbb{S}_{++}^n$, $\bm{X}\neq\bm{Y}$, $\forall t\in(0,1)$, 
    \begin{align}
        &\mathrm{tr}\left(\left(t\bm{X}+(1-t)\bm{Y}\right)^\frac{1}{2} \right) \nonumber\\
        &- \left(t\mathrm{tr}\left(\bm{X}^{\frac{1}{2}}\right)+(1-t)\mathrm{tr}\left(\bm{Y}^{\frac{1}{2}}\right)\right)> 0.
    \end{align}
\end{lem}

\begin{lem}[Ch.~3~\cite{Lipeng2013}]\label{lem:bures_min}
    The BW distance satisfies
    \begin{subequations}
    \begin{align}
        d_{BW}^2(\bm{X},\bm{Y})=\min_{U}&~\mathrm{tr}\left(\bm{X}+\bm{Y}-2\bm{U}\right),\\
        \mathrm{s.t.}&~\begin{bmatrix}
        \bm{X}&\bm{U}\\
        \bm{U}^\top&\bm{Y}
        \end{bmatrix}\succeq 0.
        \end{align}
    \end{subequations}
\end{lem}

\begin{lem}[Ch.~11~\cite{RockafellarWets1998}]\label{lem:minmin}
Let $X$ and $Y$ be nonempty sets, and let $f:X\times Y\to\mathbb{R}$.
Assume that the minimum exists. Then
\begin{align}
\min_{(x,y)\in X\times Y} f(x,y)
=
\min_{x\in X}\left(\min_{y\in Y} f(x,y)\right).
\end{align}
\end{lem}

\end{document}